\documentclass[10pt]{article}
\usepackage{amsmath, amsthm}\usepackage{enumerate}\usepackage{amssymb}\usepackage{bbold}\usepackage{color}
\usepackage{bbm}\usepackage{dsfont}\usepackage{color}\usepackage{xcolor}
\newtheorem{theorem}{Theorem}[section]
\newtheorem{proposition}[theorem]{Proposition}

\newtheorem{example}[theorem]{Example}
\newtheorem{corollary}[theorem]{Corollary}
\newtheorem{remark}[theorem]{Remark}

\newtheorem{definition}[theorem]{Definition}

\DeclareMathOperator{\dc}{\xrightarrow[]{D_{st_o}(p,q)}}

\DeclareMathOperator{\sd}{\downarrow^{st_o}}

\DeclareMathOperator{\soc}{\xrightarrow[]{st_o}}

\DeclareMathOperator{\co}{\xrightarrow[]{o}}

\begin{document}

\title{Deferred statistical order convergence in Riesz spaces}
\maketitle\author{\centering{{Mehmet Küçükaslan$^1$ and Abdullah Ayd\i n$^{2,*}$\vspace{2mm} \\
			\small $^1$Department of Mathematics, Mersin University, Mersin, Turkey \\  \small mkkaslan@gmail.com\\
			\small $^2$Department of Mathematics, Mu\c{s} Alparslan University, Mu\c{s}, Turkey \\  \small a.aydin@alparslan.edu.tr\\$^*$Corresponding Author

\abstract
Some types of statistical convergence such as statistical order and deferred statistical convergences have been studied and investigated in Riesz spaces, recently. In this paper, we introduce the concept of deferred statistical convergence in Riesz spaces with order convergence. Moreover, we give some relations between deferred statistical order convergence and other kinds of statistical convergences.
\endabstract


\section{Introduction and Preliminaries}\label{sec:section1}

Statistical convergence is a generalization of the ordinary convergence of a real or complex sequence. It was introduced by Steinhaus in \cite{St}. Maddox discussed the statistical convergence in more general abstract spaces such as locally convex spaces in \cite{M}. K\"{u}\c{c}\"{u}kaslan and Y\i lmazt\"{u}rk introduced and investigated the deferred statistical convergence in \cite{KY}. It is enough to mention the theory of statistical convergence (cf. \cite{Agnew,Fast,Fridy,KDD,M}). On the other hand, Riesz space (or, vector lattice) is another concept of functional analysis that was introduced by Riesz \cite{Riez}. Then, many authors developed the subject. Riesz space is an ordered vector space that has many applications in measure theory, Banach space, operator theory, and applications in economics (cf. \cite{AB,ABPO,AAydn2,LZ,Za}). The present paper aims to combine the concepts of deferred statistical convergence of real sequences and order convergence in Riesz spaces.

A real-valued vector space $E$ with an order relation is said to be {\em ordered vector space} if, for each $x,y\in E$ with $x\leq y$, we have $x+z\leq y+z$ and $\alpha x\leq\alpha y$ for all $z\in E$ and $\alpha \in \mathbb{R}_+$. An ordered vector space $E$ is called {\em Riesz space} or {\em vector lattice} if, for any two vectors $x,y\in E$, the infimum and the supremum
$$
x\wedge y=\inf\{x,y\} \ \ \text{and} \ \ x\vee y=\sup\{x,y\}
$$
exist in $E$, respectively. For an element $x$ in a vector lattice $E$, {\em the positive part}, {\em the negative part}, and {\em module} of $x$ are respectively defined as follows:
$$
x^+:=x\vee0, \ \ \ x^-:=(-x)\vee0\ \ and \ \ |x|:=x\vee (-x).
$$ 
Thus, in the present paper, the vertical bar $|\cdot|$ of elements in vector lattices will stand for the module of the given elements. A subset $A$ of a vector lattice $E$ is called \textit{solid} if, for each $x\in A$ and $y\in E$,  $|y|\leq|x|$ implies $y\in A$. A solid vector subspace of a vector lattice is referred to as an \textit{ideal}. A vector lattice is called {\em $\sigma$-order complete} if every nonempty bounded above countable subset has a supremum (or, equivalently, whenever every nonempty bounded below countable subset has an infimum) (cf. \cite{AB}).

A sequence $(x_n)$ in a Riesz space $E$ is said to be increasing whenever $x_1 \leq x_2\leq\cdots$ and is decreasing if $x_1 \geq x_2\geq \cdots$ holds. Then, we denote them by $x_n\uparrow$ and $x_n\downarrow$, respectively. Moreover, if $x_n\uparrow$ and $\sup x_n=x$, then we write $x_n\uparrow x$. Similarly, if $x_n\downarrow$ and $\inf x_n=x$, then we write $x_n\downarrow x$. Then, we call that $(x_n)$ is increasing or decreasing as monotonic. On the other hand, order convergence is crucial for this paper, and so, we continue with its definition.
\begin{definition}\label{def of order and un-order conv}
Let $(x_n)$ be a sequence in a vector lattice $E$. Then, it is called {\em order convergent} to $x\in E$ if there exists another sequence $y_n\downarrow 0$ (i.e., $\inf y_n=0$ and $y_n\downarrow$) such that $|x_n-x|\le y_n$ holds for all $n\in \mathbb{N}$, and abbreviated it as $x_n\co x$.
\end{definition}

For the definition of statistical convergence, the important point is the natural density of subsets of natural numbers. Recall that the \emph{density of} a subset $K$ of $\mathbb{N}$ is the limit $\lim_{n\rightarrow \infty}\frac{1}{n}\left\vert\left\{k\leq n:k\in K\right\}\right\vert$  whenever this unique limit exists, and it is mostly abbreviated by $\delta(K)$, where $\left\vert\left\{k\leq n:k\in K\right\}\right\vert$ is the cardinality of $K$, and it does not exceed $n$. A sequence $(x_n)$ of real numbers is called {\em statistical convergent} to a real number $l$ if, for every $\varepsilon>0$, we have 
$$
\lim\limits_{n\to\infty}\frac{1}{n}\big\lvert\{k:n\geq k, \ |x_n-l|>\varepsilon\}\big\rvert=0.
$$

Now, consider a sequence $x:=(x_k)$ and take the sequences $(p_n)$ and $(q_n)$ of non-negative integers such that $p_n<q_n$ for each $n$ and $q_n$ divergent to infinity. Then, define a new sequence
$$
(D_{p,q}x)_n:=\frac{1}{q_n-p_n}\sum_{k=p_n+1}^{q_n}x_k,
$$
for each $n\in\mathbb{N}$. The sequence $(D_{p,q}x)_n$ is called the {\em deferred Ces\'{a}ro mean} as a generalization of Ces\'{a}ro mean of real (or complex) valued sequence; see \cite{Agnew}. On the other hand, $x$ is said to be {\em strong $D_{p,q}$-convergent} to $l$ if the following limit exists
$$
\lim\limits_{n\to\infty}\frac{1}{q_n-p_n}\sum_{k=p_n+1}^{q_n}|x_k-l|=0.
$$
Then, we abbreviate it as $x_k\xrightarrow{D[p,q]}l$. In this article, unless otherwise, when we mention $p$ and $q$ sequences, they always hold the above properties, and also, these properties are said to be {\em the deferred property}.

A sequence $x:=(x_k)$ is called {\em deferred statistical convergent} to $l\in\mathbb{R}$ whenever, for all $\varepsilon>0$, we have
$$
\lim\limits_{n\to\infty}\frac{1}{q_n-p_n}\big|\{p_n<k\leq q_n:|x_k-l|\geq\varepsilon\}\big|=0
$$
holds; see \cite{KY}. In this case, we write $x_k\xrightarrow{DS[p,q]}l$. 

A characterization of statistical convergence on vector lattices was introduced by Ercan in \cite{ZE}, and also, some kinds of statistical convergence in Riesz spaces were introduced and studied by Ayd\i n in \cite{AAydn1,AAydn3,AAydn4,AE}.
\begin{definition}
Let $(x_n)$ be a sequence in a Riesz space $E$. Then, $(x_n)$ is called
\begin{enumerate}
\item[-] {\em statistical order decreasing} to $0$ if there exists a set $K=\{k_1<k_2<\cdots\}\subset\mathbb{N}$ with $\delta(K)=1$ such that $(x_{k_n})$ is decreasing and $\inf\limits_{n\in K}(x_{k_n})=0$, i.e., $(x_{k_n})_{k_n\in K}\downarrow 0$, and it is abbreviated as $x_n\sd 0$;
	
\item[-] {\em statistical order convergent} to $x\in E$ if there exists a sequence $q_n\sd0$ with an index set $K=\{k_1<k_2<\cdots\}\subset\mathbb{N}$ such that $\delta(K)=1$ and 
$$
|x_{k_n}-x|\leq q_{k_n}
$$
for every $k_n\in K$, and so, we write $x_n\soc x$.
\end{enumerate}
\end{definition}
It is clear that every order convergent sequence is statistical order convergent to the same point.
%%%%%%%%%%%%%%%%%%%%%%%%%%%%%%%%%%%%%%%%%%%%%%%%%%%%%%%%%%%%%%%%%%%%%%%
\section{Deferred statistical decreasing}
Tripathy \cite{Trip} introduced the statistical monotonicity for real sequences, and also, statistically monotone sequences in Riesz spaces were investigated. We extend it to deferred statistical decreasing in Riesz spaces.
\begin{definition}\label{first def.}
Let $(p_n)$ and $(q_n)$ be sequences of nonnegative integers satisfying the deferred property. Then, a sequence $(z_n)$ in a Riesz space $E$ is called {\em deferred statistical order decreasing} to $0$ if there exists a set $K\subseteq\mathbb{N}$ such that the deferred density of $K$
$$
\delta_{p,q}(K):=\lim\limits_{n\to\infty}\frac{1}{q_n-p_n}\big|\{p_n< k\leq q_n:k\in K\}\big|=1
$$
and $(z_{k_n})_{k_n\in K}\downarrow0$ holds on $K$. Then, we abbreviate it as $z_n\downarrow^{D_{st_o}}_{p,q}0$.
\end{definition}

\begin{remark}\label{remark1} \
\begin{enumerate}[(i)]
\item If $q(n)=n$ and $p(n)=0$, then Definition \ref{first def.} coincides with the definition of statistical order decreasing.

\item If $(z_n)$ is monotone decreasing to zero, then it is deferred statistical order decreasing to zero. But, the converse does not need to be true in general. To see this, consider the Euclidean space $\mathbb{R}^2$ with the coordinatewise ordering and the sequences $q(n)=n$ and $p(n)=0$ and $(z_n)$ denoted by
$$
z_n:=
\begin{cases} 
(0,n^2) & \text{if} \ n=k^3 \\
(0,\frac{1}{n^2}) & \text{if} \ n\neq k^3
\end{cases},
$$
where $k\in\mathbb{N}$. Hence, we get $z_n\downarrow^{D_{st_o}}_{p,q} (0,0)$. But, observe that the whole sequence $(z_n)$ is not monotonic.

\item A deferred statistical order decreasing to zero sequence may contain a subsequence of decreasing or incomparable elements of $E$ but the index set of such a subsequence has deferred density zero.

\item In Riesz spaces, it is well known that $z_n\downarrow0$ implies $z_{k_n}\downarrow 0$ for every subsequence $(z_{k_n})$ of $(z_n)$. However, this may
not hold in the setting of deferred statistical monotone decreasing sequences. For example, take the sequences in the $(ii)$ with the subsequence $(z_{k_n})$, where $k_n=j^3$ for some $j\in\mathbb{N}$, does not have a supremum.
\end{enumerate}
\end{remark}

In the general case, the example in Remark \ref{remark1}$(iv)$ shows that a subsequence of deferred statistical monotone decreasing sequence need not be deferred statistical monotone decreasing. But, we give a positive result in the following theorem.
\begin{theorem}
Let $(z_n)$ be a sequence in a Riesz space $E$. If $z_n\downarrow^{D_{st_o}}_{p,q}0$, then any subsequence $(z_{k_n})$ of $(z_n)$ with index set $\delta_{p,q}(K)=1$ such that $(z_{k_n})$ is decreasing on $K$ is deferred statistical order decreasing to $0$.
\end{theorem}

\begin{proof}
Suppose that $z_n\downarrow^{D_{st_o}}_{p,q}0$ holds in $E$. Then, there exists a set $K\subset\mathbb{N}$ such that $\delta_{p,q}(K)=1$ and $(z_{k_n})_{k_n\in K}\downarrow0$ on $K$. Let us consider any arbitrary index set $M\subseteq\mathbb{N}$ such that $K\neq M$, $\delta_{p,q}(M)=1$ and $(z_n)$ is decreasing on $M$. It can be observed that if there is not such a set $M$, then the poof is complete. It follows from $z_{k_n}\downarrow0$ that $0\leq z_{k_n}$ for all $k_n\in K$. Also, we have $\delta_{p,q}(K\cap M)=1$. Thus, for some $k_m\in K$ and $m_n\in M$, we have $k_n=m_n$. Hence, we have $z_{m_1}\geq z_{m_2}\geq\cdots\geq z_{m_n}=z_{k_n}\geq 0$. We can find infinitely many of such pair of indices. By continuing this way, we obtain $z_{m_n}\geq 0$ for every $m_n\in M$, i.e., zero is a lower bound of $(z_{m_n})$. Take another lower bound $u$ of $(z_{m_n})$. Therefore, we have $u\leq z_{m_n}$ for every $m_n\in M$. Then, we can find some $z_{n_{k_t}}$ such that $z_{n_{k_t}}=z_{m_k}\geq u$ for some $m_k\in M$. By following this way, we can construct a subsequence $(z_{n_{k_1}},z_{n_{k_2}},\cdots)$ of $(z_{k_n})$ such that $u$ is a lower bound of $(z_{n_{k_t}})$ for $t\in \mathbb{N}$. It follows from $z_{k_n}\downarrow 0$ that the infimum of every subsequence of $(z_{k_n})$ is zero. Hence, we get
$u=0$. Therefore, we get the desired result, $z_{m_n}\downarrow^{D_{st_o}}_{p,q}0$.
\end{proof}

In the next results without proof, we give the linear property of deferred statistical order decreasing sequences.
\begin{proposition}\label{uo-conv implies stuo-conv.}
Let $x_n\downarrow^{D_{st_o}}_{p,q}0$ and $y_n\downarrow^{D_{st_o}}_{p,q}0$ be a sequence in a Riesz space $E$ and $\lambda\in\mathbb{R}$. Then, we have
\begin{enumerate}
\item[(i)] $(x_n+y_n)\downarrow^{D_{st_o}}_{p,q}0$;
\item[(ii)] $\lambda x_n\downarrow^{D_{st_o}}_{p,q}0$.
\end{enumerate}
\end{proposition}

%%%%%%%%%%%%%%%%%%%%%%%%%%%%%%%%%%%%%%%%%%%%%%%%%%%%%%%%%%%%%%%%%%%%%%%%%5
\section{Deferred statistical order convergence}

\begin{definition}\label{main def}
Let $p$ and $q$ be sequences of positive integers satisfying the deferred property. Then, a sequence $(x_n)$ in a Riesz space $E$ is called {\em deferred statistical order convergent} to $x$ if there exists a sequence $z_n\downarrow^{D_{st_o}}0$ with an index set $K\subseteq\mathbb{N}$ such that $\delta_{p,q}(K)=1$ and
$$
|x_{k_n}-x|\leq z_{k_n}
$$
holds for all $k_n\in K$. Then, we write $x_n\dc x$.
\end{definition}

\begin{remark}
It can be seen that, in the case of $x_n\dc x$, we have
$$
\delta_{p,q}(\{n\in\mathbb{N}:|x_n-x|\nleq z_n\})=0.
$$
\end{remark}

\begin{remark}
It can be observed that the deferred statistical order convergence of the sequence $(x_n)$ in Definition \ref{main def} with sequence $(z_n)$ to $x$ implies that $x_{k_n}\dc x$ with the same sequence $(z_n)$. The converse is also true, i.e., if there exists a subsequence $(x_{k_n})\dc x$ of a sequence $(x_n)$ with a sequence $z_n\downarrow^{D_{st_o}}0$, then $x_n\dc x$ with the same sequence $(z_n)$.
\end{remark}

It is clear that deferred statistical order decreasing sequence is deferred statistical order convergent. But, the converse does not hold in general.

\begin{remark}\label{order and deferred}
Let $q(n)=n$ and $p(n)=0$. Then, we have the following observations:
\begin{enumerate}
\item[(i)] an order convergent sequence is deferred statistical order convergent to its order limit;

\item[(ii)] the statistical order convergence and deferred statistical order convergence coincide;
\end{enumerate}
\end{remark}

One can observe that a subsequence of a deferred statistical order convergent sequence need not be deferred statistical order convergent.

\begin{proposition}
Let $(x_n)$ be a sequence in a Riesz space $E$. Then, $x_n\dc x$ holds if and only if there exists another sequence $(y_n)$ in $E$ such that $\delta_{p,q}(\{n\in\mathbb{N}:x_n=y_n\})=1$ and $y_n\dc x$.
\end{proposition}

\begin{proof}
Suppose that there exists a sequence $(y_n)$ in $E$ such that $\delta_{p,q}(\{n\in\mathbb{N}:x_n=y_n\})=1$ and $y_n\dc x$. Then, there is another sequence $z_n\downarrow^{D_{st_o}}0$ in $E$ with $\delta_{p,q}(K)=1$ such that
$|x_{k_n}-x|\leq z_{k_n}$ for each $k_n\in K$. Thus, it follows from the including 
$$
\{p_n+1\leq m\leq q_n:|x_m-x|\nleq z_m\}\subseteq \{p_n+1\leq m\leq q_n:x_m\neq y_m\}
$$
$$
\cup \{p_n+1\leq m\leq q_n:|y_m-x|\nleq z_m\}
$$ 
that we have 
$$\lim\limits_{n\to\infty}\frac{1}{q_n-p_n}|\{p_n+1\leq m\leq q_n:|x_m-x|\nleq z_m\}|
$$
$$
\leq \lim\limits_{n\to\infty}\frac{1}{q_n-p_n}|\{p_n+1\leq m\leq q_n:x_m\neq y_m\}|
$$
because of $\delta_{p,q}(\{p_n+1\leq m\leq q_n:|y_m-x|\nleq z_m\})=0$. Thus, we obtain 
$$
\lim\limits_{n\to\infty}\frac{1}{q_n-p_n}|\{p_n+1\leq m\leq q_n:|x_m-x|\nleq z_m\}|=0.
$$
Therefore, we get the desired result, $x_n\dc x$. The other part of proof is obvious, and so, we omit it.
\end{proof}

\begin{proposition}\label{linear and unit}
The deferred statistical order limit is linear and uniquely determined.
\end{proposition}

\begin{proof}
Assume that $x_n\dc x$ and $x_n\dc y$ hold in a Riesz space $E$. Then, there are sequences $z_n\downarrow^{D_{st_o}}0$ with $\delta_{p,q}(K)=1$ and $t_n\downarrow^{D_{st_o}}0$ with $\delta_{p,q}(M)=1$ such that $|x_{k_n}-x|\leq z_{k_n}$ and $|x_{m_n}-y|\leq t_{m_n}$ for all $k_n\in K$ and $m_n\in M$. Thus, it follows that
$$
|x-y|\leq |x-x_{j_n}|+|x_{j_n}-y|\leq z_{j_n}+t_{j_n}
$$
for every $j_n\in J:=K\cap M$. By using $(z_{j_n}+t_{j_n})_{j_n\in J}\downarrow0$, we obtain that $ |x-y|=0$. Thus, we get the equality of $x$ and $y$.

Now, for the linearity of the deferred statistical order limit, take sequences $x_n\dc x$ and $y_n\dc y$ in a Riesz space $E$. Then, there are sequences $z_n\downarrow^{D_{st_o}}0$ and $t_n\downarrow^{D_{st_o}}0$ such that $\delta_{p,q}(\{n\in \mathbb{N}:|x_n-x|\nleq z_n\})=0$ and $\delta_{p,q}(\{n\in \mathbb{N}:|y_n-y|\nleq t_n\})=0$. It follows from the triangular inequality in Riesz spaces that 
$$
\{n\in \mathbb{N}:|(x_n+y_n)-(x+y)|\nleq z_n+t_n\}\subseteq \{n\in \mathbb{N}:|x_n-x|\nleq z_n\}
$$
$$
\cup\{n\in \mathbb{N}:|y_n-y|\nleq t_n\}.
$$
Therefore, we obtain $\delta_{p,q}(\{n\in \mathbb{N}:|(x_n+y_n)-(x+y)|\nleq z_n+t_n\})=0$, i.e., we obtain $x_n+y_n\dc x+y$.
\end{proof}

In the following result, we observe some relations between deferred statistical order convergence and lattice properties.
\begin{theorem}
Let $x_n\dc x$ and $y_n\dc y$ in a Riesz space $E$. Then, we have the following statement:
\begin{enumerate}[(i)]
\item $x_n\vee y_n\dc x\vee y$;
\item $x_n\wedge y_n\dc x\wedge y$;
\item $x_n^+\dc x^+$;
\item $x_n^-\dc x^-$;
\item $|x_n|\dc |x|$.
\end{enumerate}
\end{theorem}

\begin{proof}
It is enough to show the first statement because the other case can be obtained by applying \cite[Thm.1.7]{ABPO} and Proposition \ref{linear and unit}. Now, from $x_n\dc x$ and $y_n\dc y$, we have sequences $z_n\downarrow^{D_{st_o}}0$ and $t_n\downarrow^{D_{st_o}}0$ with indexes sets $\delta_{p,q}(K)=\delta_{p,q}(M)=1$ such that $|x_{k_n}-x|\leq z_{k_n}$ and $|y_{m_n}-y|\leq t_{m_n}$ hold for all $k_n\in K$ and $m_n\in M$. On the other hand, by applying \cite[Thm.1.9]{ABPO} and by taking $J:=N\cap M$, we get
$$
\lvert x_{j_n}\vee y_{j_n}-x\vee y\rvert\leq \lvert x_{j_n}-x\rvert+\lvert y_{j_n}-y\rvert\leq z_{j_n}+t_{j_n}
$$
for every $j_n\in J$. Hence, we obtain 
$$
\delta_{p,q}(\{n\in \mathbb{N}:\lvert x_{j_n}\vee y_{j_n}-x\vee y\rvert\nleq z_{j_n}+t_{j_n}\})=0.
$$
Therefore, we get the desired result, $x_n\vee y_n\dc x\vee y$.
\end{proof}

\begin{corollary}\label{corollary}
The positive cone $E_+=\{x\in E:0\leq x\}$ of a Riesz space $E$ is closed under the deferred statistical order convergence.
\end{corollary}

\begin{proposition}
If $x_n\dc x$, $y_n\dc y$ and $x_n\ge y_n$ satisfy for every $n\in\mathbb{N}$ in a Riesz space, then $x\ge y$ holds.
\end{proposition}

\begin{proof}
Assume that $y_n\leq x_n$ holds for each $n\in\mathbb{N}$. Then, we have $0\leq x_n-y_n \in E_+$ for each $n\in\mathbb{N}$. It follows from Corollary \ref{corollary} that $x_n-y_n\dc x-y\in E_+$ because of $(x_n-y_n)\in E_+$. Thus, we get $x-y\geq 0$, i.e., $x\geq y$.
\end{proof}

\begin{theorem}\label{motone and st implies order}
If $(x_n)$ is a monotone and deferred statistical order convergent in a Riesz space, then it is order convergent.
\end{theorem}

\begin{proof}
Suppose that $(x_n)\downarrow$ and $x_n\dc x$ in a Riesz space $E$. Fix any $k\in\mathbb{N}$. Then, we have $x_k-x_n\geq 0$ for all $n\geq k$. It follows that $x_k-x_n\dc x_k-x\geq 0$, i.e., $x_k\geq x$. Thus, $x$ is an lower bound of $(x_n)$ because $k$ is arbitrary. Choose another lower bound $z$ of $(x_n)$. Hence, we have $x_n-z\dc x-z\geq 0$, i.e., $x\geq z$. Therefore we get the desired result, $x_n\downarrow x$.
\end{proof}

\begin{remark}\label{ideal and vector lattice}
Let $A$ be an ideal in a vector lattice $E$ and $(a_n)$ be a sequence in $A$. One can observe that if $a_n\co 0$ in $A$, then $a_n\co 0$ in $E$. Hence, it is clear that $a_n\downarrow^{D_{st_o}} 0$ in $A$ implies $a_n\downarrow^{D_{st_o}} 0$ in $E$. For the converse, if $a_n\co 0$ in $E$ and order bounded, then $a_n\co 0$ in $A$, and so, $a_n\downarrow^{D_{st_o}} 0$ in $E$ implies $a_n\downarrow^{D_{st_o}} 0$ in $A$ for order bounded sequences.
\end{remark}

Thanks to Remark \ref{ideal and vector lattice}, we give the following two results.
\begin{theorem}
Let $A$ be an ideal in an $\sigma$-order complete vector lattice and $(x_n)$ be a sequence in $A$. Then, $x_n\dc 0$ in $A$ if and only if $x_n\dc 0$ in $E$.
\end{theorem}

\begin{proof}
Assume that $x_n\dc 0$ in $A$. Then, there exists a sequence $z_n\downarrow^{D_{st_o}}0$ in $A$ with index set $\delta_{p,q}(K)=1$ such that $|x_{k_n}|\leq z_{k_n}$ for all $k_n\in K$. Now, by using Remark \ref{ideal and vector lattice}, it follows from $(z_{k_n})_{k_n\in K}\downarrow0$ in $A$ that $(z_{k_n})_{k_n\in K}\downarrow0$ in $E$, i.e., we get $z_n\downarrow^{D_{st_o}}0$ in $E$ Therefore, we have $x_n\dc 0$ in $E$.

Conversely, assume $x_n\dc0$ in $E$. Then, there is a sequence $z_n\downarrow^{D_{st_o}}0$ in $E$ with index set $\delta_{p,q}(K)=1$ such that $|x_{k_n}|\leq z_{k_n}$ for all $k_n\in K$. Thus, Remark \ref{ideal and vector lattice} implies that $z_n\downarrow^{D_{st_o}}0$ in $A$. Therefore, we get $x_n\dc 0$ in $A$.
\end{proof}

The following result is similar to \cite[Thm.3.1.]{KDD}.
\begin{theorem}
Let $(x_n)$ be a sequence in a Riesz space $E$ and $(x_{k_n})_{k_n\in K}$ be a subsequence of $(x_n)$. If the limit
$$
\lim\inf\limits_{n\to\infty}\frac{1}{q_n-p_n}\big|\{p_n<k_n\leq q_n:k_n\in K\}\big|>0
$$
holds and $x_n\dc x$ for some sequences $p$ and $q$ satisfying the deferred property, then $x_{k_n}\dc x$. 
\end{theorem}

\begin{proof}
Assume that $x_n\dc x$ satisfies in $E$. Then, there is a sequence $z_n\downarrow^{D_{st_o}}0$ in $E$ such that $\delta_{p,q}(\{n\in\mathbb{N}:|x_n-x|\nleq z_n\})=0$. It can be seen that
$$
\{p_n<k_n\leq q_n:k_n\in K,\ |x_{k_n}-x|\nleq z_n\}\subseteq \{p_n<n\leq q_n:|x_n-x|\nleq z_n\}.
$$
Thus, by taking $H_n:=\{p_n<k_n\leq q_n:k_n\in K\}$ for all $n$, we have 
$$
\frac{1}{|H_n|}\big|\{p_n<k_n\leq q_n:k_n\in K,\ |x_{k_n}-x|\nleq z_n\}\big|
$$
$$
\leq\frac{1}{|H_n|}\{p_n<n\leq q_n:|x_n-x|\nleq z_n\}.
$$
Therefore, it is enough to show $\lim\sup\limits_{n\to\infty}\frac{1}{|H_n|}\big|\{p_n<n\leq q_n:|x_n-x|\nleq z_n\}\big|=0$ for proving the convergence $x_{k_n}\dc x$. We observe the following inequality
$$
\lim\inf\limits_{n\to\infty}\frac{|H_n|}{q_n-p_n}\lim\sup\limits_{n\to\infty}\frac{|\{p_n<n\leq q_n:|x_n-x|\nleq z_n\}|}{|H_n|}
$$
$$
\leq \lim\sup\limits_{n\to\infty}\frac{|\{p_n<n\leq q_n:|x_n-x|\nleq z_n\}|}{q_n-p_n}.
$$
Therefore, we have 
$$
\lim\sup\limits_{n\to\infty}\frac{1}{|H_n|}\big|\{p_n<n\leq q_n:|x_n-x|\nleq z_n\}\big|=0
$$
because of $x_n\dc x$. Thus, we obtain the desired result.
\end{proof}

In Remark \ref{order and deferred}, we give a relation between statistical  order and deferred statistical order convergences by taking $q(n)=n$ and $p(n)=0$. We give another relation under a new condition in the next theorem.
\begin{theorem}\label{bounded}
If the sequence $\big(\frac{p_n}{q_n-p_n}\big)$ is bounded for any $p$ and $q$ sequences having the deferred property, then the statistical order convergence implies the deferred statistical order convergence.
\end{theorem}

\begin{proof}
Assume that $x_n\soc x$ in a Riesz space $E$ and $\big(\frac{p_n}{q_n-p_n}\big)$ is a bounded sequence for some sequences $p$ and $q$ satisfying the deferred property. Thus, there exists a sequence $z_n\sd0$ such that  
$$
\lim\limits_{n\to\infty}\frac{1}{n}|\{k\leq n:|x_k-x|\nleq z_k\}|=0.
$$
It follows from the deferred properties of $(q_n)$ that we obtain
$$
\lim\limits_{n\to\infty}\frac{1}{q_n}|\{k\leq q_n:|x_k-x|\nleq z_k\}|=0.
$$
So, by the following inclusion 
$$
\{p_n< k\leq q_n:|x_k-x|\nleq z_k\}\subseteq \{k\leq q_n:|x_k-x|\nleq z_k\}, 
$$
we obtain
$$
\lim\limits_{n\to\infty}\frac{1}{q_n-p_n}|\{p_n< k\leq q_n:|x_k-x|\nleq z_k\}|
$$
$$\ \ \ \ \ \ \ \ \ \ \ \ \ \ \ \ \ \ \ \ \leq \lim\limits_{n\to\infty} \frac{1}{q_n}(1+\frac{p_n}{q_n-p_n})|\{k\leq q_n:|x_k-x|\nleq z_k\}|.
$$
Thus, we get the desired result, $x_n\dc x$.
\end{proof}

The converse of Theorem \ref{bounded} need not be true in general. To see this, we give the following example.
\begin{example}
Consider the Riesz space $E:=\mathbb{R}^2$ equipped with the coordinatewise ordering and a sequence $(x_n)$ in $E$ as follows:
$$
x_n:=
\begin{cases} 
(0,\frac{n+1}{2}), & \text{n is odd} \\
(0,-\frac{n}{2}), & \text{n is even}
\end{cases}
$$
for all $n$. Also, take $(q_n):=(2n)$ and $(p_n):=(4n)$. Then, it is clear that the assumption of Theorem \ref{bounded} is fulfilled, and also, $x_n\dc (0,0)$. But, it is not statistical order convergent.
\end{example}

\begin{theorem}
Let $p',q'$ and $p,q$ be pairs of sequences satisfying the deferred property such that $p_n\leq p'_n$ and $q_n'\leq q_n$ for each $n\in\mathbb{N}$, and $(x_n)$ be a sequence in a Riesz space $E$. Then, $x_n\xrightarrow[]{D_{st_o}(p',q')}x$ implies $x_n\dc x$ in $E$ whenever the sets $\{k:p_n<k\leq p'_n\}$ and $\{k:q'_n<k\leq q_n\}$ are finite for every $n\in\mathbb{N}$.
\end{theorem}

\begin{proof}
Assume that $x_n\xrightarrow[]{D_{st_o}(p',q')}x$ holds in $E$. Then, there exists sequence $z_n\downarrow^{D_{st_o}}0$ such that
$$
\delta_{p,q}(\{n\in\mathbb{N}:|x_n-x|\nleq z_n\})=0.
$$
On the other hand, we have the following equality
$$
\{k:p_n<k\leq q_n, \ |x_n-x|\nleq z_n\}=\{k:p_n<k\leq p'_n, \ |x_n-x|\nleq z_n\}
$$
$$
\ \ \ \ \ \cup\{k:p'_n<k\leq q'_n, \ |x_n-x|\nleq z_n\} \cup\{k:q'_n<k\leq q_n, \ |x_n-x|\nleq z_n\}.
$$
It follows that
\begin{eqnarray*}
\lim\limits_{n\to\infty}\frac{1}{q_n-p_n}|\{k:p_n<k\leq q_n, \ |x_n-x|\nleq z_n\}|=0.
\end{eqnarray*}
Hence, we get $x_n\dc x$.
\end{proof}

\begin{corollary}
Let $p',q'$ and $p,q$ be pairs of sequences satisfying the deferred property such that $\lim\limits_{n\to\infty}\frac{q_n-p_n}{q'_n-p'_n}=t>0$, and $(x_n)$ be a sequence in a Riesz space $E$. Then, $x_n\xrightarrow[]{D_{st_o}(p',q')}x$ implies $x_n\dc x$ in $E$.
\end{corollary}

Now, take the set $C^{p,q}_{(z_n)}:=\{(x_n):\exists x\in E,x_n\dc x\ \text{with} \ (z_n)\}$ for a fixed sequence $z_n\downarrow^{D_{st_o}}0$. Then, it is clear that $C^{p,q}_{(z_n)}\subseteq C^{p,q}_{(w_n)}$ whenever $z_n\leq w_n$ holds for all $n\in\mathbb{N}$.
\begin{proposition}
If $\delta_{p,q}(\{n\in\mathbb{N}:z_n\neq w_n\})=0$, then $C^{p,q}_{(z_n)}=C^{p,q}_{(w_n)}$.
\end{proposition}

\begin{proof}
Suppose that $\delta_{p,q}(\{n\in\mathbb{N}:z_n\neq w_n\})=0$ holds for some sequences $z_n\downarrow^{D_{st_o}}0$ and $w_n\downarrow^{D_{st_o}}0$. Take $(x_n)\in C^{p,q}_{(z_n)}$. Then, we have $\delta_{p,q}(\{n:|x_n-x|\nleq z_n\})=1$. It follows from the following inclusion
\begin{eqnarray*}
\{n:|x_n-x|\nleq w_n\}\subseteq\{n:|x_n-x|\nleq z_n\}\cup\{n:z_n\neq w_n\}
\end{eqnarray*}
that we obtain $(x_n)\in C^{p,q}_{(w_n)}$, and so, $C^{p,q}_{(z_n)}\subseteq C^{p,q}_{(w_n)}$. Similarly, we can get $C^{p,q}_{(w_n)}\subseteq C^{p,q}_{(z_n)}$. Therefore, we obtain $C^{p,q}_{(z_n)}=C^{p,q}_{(w_n)}$.
\end{proof}

It is clear that $C^{p,q}_{(z_{k_n})}\subseteq C^{p,q}_{(z_n)}$ for any subsequence $z_{k_n}\downarrow^{D_{st_o}}0$ of sequence $z_n\downarrow^{D_{st_o}}0$. For the converse, we give the following result.
\begin{proposition}
If $z_n\downarrow 0$ holds, then $C^{p,q}_{(z_n)}\subseteq C^{p,q}_{(z_{k_n})}$ satisfies for each subsequence $(z_{k_n})_{k_n\in K}$ of $(z_n)$ with $\delta_{p,q}(K)=1$.
\end{proposition}
\begin{proof}
Assume that $z_n\downarrow 0$ and $(z_{k_n})_{k_n\in K}$ be a subsequence of $(z_n)$ with $\delta_{p,q}(K)=1$. Take any element $(x_n)\in C^{p,q}_{(z_n)}$. Then, we have $x\in E$ and an index set $\delta_{p,q}(M)=1$ such that $|x_{m_n}-x|\leq z_{m_n}$ holds for all $m_n\in M$. On the other hand, by taking $J:=M\cap K$, we have $|x_{j_n}-x|\leq z_{j_n}$ for each $j_n\in J$. Since $(z_{j_n})$ is a subsequence of $(z_{k_n})$, we obtain $(x_n)\in C^{p,q}_{(z_{k_n})}$.
\end{proof}
%%%%%%%%%%%%%%%%%%%%%%%%%%%%%%%%%%%%%%%%%%%%%%%%%%%%5
{\tiny 

}

\begin{thebibliography}{30}

	\bibitem{Agnew}
	R. P. Agnew, On deferred Cesaro Mean, Annals of Mathematics, 33, 413-421, 1932.

	\bibitem{AB}
	C. D. Aliprantis, O. Burkinshaw, Locally Solid Riesz Spaces with Applications to Economics, American Mathematical Society, 2003.
	
	\bibitem{ABPO}
	C. D. Aliprantis, O. Burkinshaw, Positive Operators, Springer, Dordrecht, 2006.
	
	\bibitem{AAydn1}
	A. Ayd{\i}n, The statistically unbounded $\tau$-convergence on locally solid Riesz spaces, Turkish J. Math., 44, 949-956, 2020.
	
	\bibitem{AAydn2}
	A. Ayd{\i}n, Multiplicative order convergence in $f$-algebras, Hacet. J. Math. Stat., 49, 998-1005, 2020.
	
	\bibitem{AAydn3}
	A. Ayd{\i}n, The statistical multiplicative order convergence in Riesz spaces, Facta Uni. Series: Math. Info., 36, 409-417, 2021.
	
	\bibitem{AAydn4}
	A. Ayd{\i}n, Statistically unbounded p-convergence in lattice-normed Riesz spaces, to appear Misk. Math. Note., 2022.
	
	\bibitem{AE}
	A. Ayd{\i}n, M. Et, Statistically multiplicative convergence on locally solid Riesz algebras, Turk. J. Math., 45, 1506-1516, 2021.
	
	\bibitem{ZE}
	Z.\ Ercan, A characterization of $u$-uniformly completeness of Riesz spaces in terms of statistical $u$-uniformly pre-completeness, Demon. Math., 42, 383-387, 2009.
	
	\bibitem{Fast}
	H.\ Fast, Sur la convergence statistique, Colloq. Math., 2, 3-4, 1951.
	
	\bibitem{Fridy}
	J.\ Fridy, On statistical convergence, Analysis, 5, 301-313, 1985.
	
	\bibitem{KY}	
	M. Küçükaslan, M. Yilmazturk, On deferred statistical convergence of sequences, Kyung. Math. j., 56, 357-366, 2016.
	
	\bibitem{KDD}	
	M. Küçükaslan, U. Değer, O. Dovgoshey, On statistical convergence of metric valued sequences, Ukr. Mat. Zhurn., 66, 712-720, 2014.
	
	\bibitem{LZ}
	W. A. J. Luxemburg, A. C. Zaanen, Riesz Spaces I, North-Holland Pub. Co., Amsterdam, 1971.
	
	\bibitem{M}
	I. J. Maddox, Statistical convergence in a locally convex space, Math. Proc. Cambr. Phil. Soc., 104, 141-145, 1988.
	
	\bibitem{Riez}
	F. Riesz, Sur la décomposition des opérations fonctionelles linéaires, Atti D. Congr. Inter. D. Math., Bologna, 1928.
	
	\bibitem{SAL}
	T.\ Salat, On statistically convergent sequences of real numbers, Math. Slov., 30, 139-150, 1980.
	
	\bibitem{St}
	H.\ Steinhaus, Sur la convergence ordinaire et la convergence asymptotique, Colloq. Math., 2, 73-74, 1951.
	
	\bibitem{Trip}
	B. C. Tripathy, On statistically convergent sequences, Bull. Calcutta Math. Soc., 90, 259–262, 1998.
	
	\bibitem{Za}
	A. C. Zaanen, Riesz Spaces II, North-Holland Publishing Co., Amsterdam, 1983.
\end{thebibliography}
\end{document}